\documentclass[a4,10pt,twocolumn]{scrartcl}
\usepackage[top=2.5cm, bottom=2.5cm, left=1.6cm, right=1.6cm]{geometry}

\pdfoutput=1

\input{myStandardPreamble_arXiv.tex}
\usepackage[english]{babel}
\usepackage[utf8x]{inputenc}
\usepackage{enumitem}					

\usepackage{helvet}
\usepackage[T1]{fontenc}
\usepackage{mathpazo}

\usepackage[hang]{footmisc}
\setfnsymbol{wiley}

\makeatletter
\renewcommand{\@biblabel}[1]{[#1]\hfill}
\makeatother

\makeatletter
\let\NAT@parse\undefined
\makeatother

\usepackage[format=plain,			
            labelsep=period,		
            font=small,			
            labelfont=bf,			
            skip=5pt			
            ]{caption}
            
\usepackage[final,kerning=false,protrusion=false]{microtype}

\let\tmp\newinsert
\let\newinsert\newbox
\usepackage{floatrow}
\let\newinsert\tmp
\newfloatcommand{capbtabbox}{table}[][\FBwidth]

\usepackage{authblk}

\newtheorem{theorem}{Theorem}
\newtheorem{lem}{Lemma}
\newtheorem{definition}{Definition}
\newtheorem{rem}{Remark}
\newtheorem{example}{Example}
\newtheorem{assumption}{Assumption}
\newtheorem{corollary}{Corollary}

\let\oldTheorem\theorem
\renewcommand{\theorem}{\oldTheorem\normalfont}
\let\oldLemma\lem
\renewcommand{\lem}{\oldLemma\normalfont}
\let\oldCorollary\corollary
\renewcommand{\corollary}{\oldCorollary\normalfont}
\let\oldDefinition\definition
\renewcommand{\definition}{\oldDefinition\normalfont}
\let\oldRemark\rem
\renewcommand{\rem}{\oldRemark\normalfont}
\let\oldAssumption\assumption
\renewcommand{\assumption}{\oldAssumption\normalfont}
\let\oldExample\example
\renewcommand{\example}{\oldExample\normalfont}

\addtokomafont{paragraph}{\itshape}

\newcommand{\gradx}[2]{\nabla_{#2}{#1}(#2)}

\usepackage{xfrac}
\newcommand{\nsqrth}{h^{-1/2}}
\newcommand{\sqrth}{\sqrt{h}}
\newcommand{\sqrthdz}{h^{3/2}}
\newcommand{\fJ}{\mathfrak{f}}
\newcommand{\FJ}{F}
\newcommand{\boldmin}{\mathord{\begin{tikzpicture}[baseline=0ex, line width=2, scale=0.13]	\draw (0,0.7) -- (2,0.7);\end{tikzpicture}}}
\newcommand{\thinmin}{\mathord{\begin{tikzpicture}[baseline=0ex, line width=0.5, scale=0.13]	\draw (0,0.7) -- (2,0.7);\end{tikzpicture}}}
\newcommand{\thinmins}{\mathord{\begin{tikzpicture}[baseline=0ex, line width=0.5, scale=0.13]	\draw (0,0.7) -- (0.7,0.7);\end{tikzpicture}}}
\newlength\figureheight
\newlength\figurewidth

\newcommand{\jacob}[1]{#1}

\usepackage{cases}
\usepackage{algorithm} 
\usepackage{algpseudocode}
\usepackage{bm}

\newcommand{\red}[1]{{\color{red}#1}}
\colorlet{green60}{green!60!black} 
\colorlet{orange90}{orange!90!black} 

\definecolor{green}{rgb}{0,0.5,0}

\title{\LARGE \textbf
Derivative-Free Optimization Algorithms  \\based on Non-Commutative Maps
\footnote{This article is a slightly extended version of \cite{feiling2018dfo} with an extra~\Cref{fig:proof}.}
}

\setkomafont{section}{\Large}
\setkomafont{subsection}{\large}
\setkomafont{subsubsection}{\itshape}

\begin{document}

\date{}
\author[1,2]{Jan Feiling}
\author[1]{Amelie Zeller}
\author[1]{Christian Ebenbauer}
\affil[1]{Institute for Systems Theory and Automatic Control, University of Stuttgart, Germany \protect\\ \texttt{\small $\lbrace$jan.feiling@ist,st119167@stud,ce@ist$\rbrace$.uni-stuttgart.de}}
\affil[2]{Dr. Ing. h.c. F. Porsche AG, Stuttgart, Germany \protect\\ \texttt{\small jan.feiling@porsche.de} \protect\\[1em]}

\maketitle

\begin{abstract}
\textbf{Abstract.} A novel class of derivative-free optimization algorithms is developed. 
The main idea is to utilize certain non-commutative maps in order to approximate the gradient of the objective function. Convergence properties of the novel algorithms are established and simulation examples are presented.
\end{abstract}

\section{INTRODUCTION}\label{sec:intro}
In the course of continuously increasing computational power, optimization algorithms and optimization-based control 
play more and more a central role in solving control and real-time decision making problems. In many tasks, the resulting optimization problems are very challenging, i.e., they are high-dimensional, non-convex, non-smooth, or of stochastic nature. 
Hence, improving existing optimization algorithms and developing novel algorithms is of central importance.
An interesting class of algorithms are derivative-free algorithms, which typically need only evaluations of the objective function for optimization. Hence, derivative-free optimization algorithms are simple and appealing in many challenging applications, for example in real-time optimization, where only noisy measurements and no mathematical description of the objective function are available and therefore no derivative, i.e., no gradient can be computed directly. Consequently, derivative-free optimization and adaptive control techniques are often used in this applications \cite{ariyur2003real, guay2003adaptive, tan2010extremum, atta2015extremum, poveda2017distributed,poveda2017robust}. Moreover, in the field of machine learning, derivative-free optimization problems \cite{nesterov2017random, duchi2015optimal, torczon1997convergence, spall1992multivariate} gain renewed interest, which are often tackled with so-called stochastic gradient approximation methods \cite{spall1992multivariate}.
In this paper, we propose a novel class of derivative-free discrete-time optimization algorithms.
The key characteristic of the proposed algorithms is the use of deterministic non-commutative maps to evaluate the objective function
at certain points such that a gradient descent step is approximated. In comparison to \cite{nesterov2017random, duchi2015optimal, torczon1997convergence, spall1992multivariate} our proposed method is not a randomized or stochastic algorithm.
Interestingly and to the best of our knowledge, the use of such a deterministic non-commutative scheme in discrete-time optimization algorithms has not been studied so far in the optimization literature.

In the geometric control theory literature, however, the effect of non-commutative vector fields is of fundamental importance \cite{brockett2014early}  and has been used in control design since decades. Recently, non-commutative vector fields have also been used in the design of continuous-time optimization algorithms. For example, Lie brackets, i.e., the infinitesimal characterization of non-commutativity between two vector fields, have been used to approximate gradients and to design extremum seeking algorithms for unconstrained and constrained optimization and adaptive control problems, see, e.g., \cite{durr2013lie,durr2014extremum,grushkovskaya2017class,durr2013saddle,scheinker2014extremum}. Furthermore, in \cite{michalowsky2017lie}, Lie brackets have been exploited to tackle the problem of distributed optimization over directed graphs.


The proposed class of discrete-time derivative-free algorithms are inspired by these Lie bracket based continuous-time, explorative optimization methods, but our approach is not simply a discretization of continuous-time methods.
%
%
For example, utilizing  non-commutative maps based on Euler-integration steps lead to 
approximation results which are different from continuous-time Lie bracket approximation results (see, e.g., \cite{altafini2016nonintegrable}),
yet we show how they can be used for gradient approximation.
Indeed, as shown in the paper, suitable integration schemes need to be employed in order to approximate the results known from continuous-time methods.
One advantage which arises from utilizing non-commutative maps in gradient approximation is the robustness with respect to noisy objective functions as we will see in simulations. This is not observed if finite differences are used for gradient approximations, like it was done in \cite{popovic2006extremum}, where extremum seeking and a version of \cite{spall1992multivariate} is combined.
Overall, the contributions of this work are as follows: Firstly, utilizing non-commutative maps, we 
develop novel discrete-time gradient approximation schemes and corresponding derivative-free optimization algorithms. Secondly, we perform a convergence analysis of the proposed algorithms. Finally, we study the proposed algorithm in different simulation examples.
The sequel of the paper is structured as follows: In the next section we state the problem setup and give the main idea how the discrete-time derivative-free optimization algorithms are derived. Section III contains the main lemmas and theorems providing the gradient approximation and the asymptotic convergence behavior of the algorithms. In Section IV we analyze the algorithm in various simulation examples. Finally, we give a summary and outlook of further work.

\textit{Notation.}
$\mathbb{N}_{0}$ denote the natural numbers including zero, i.e., $\{0,1,2,\ldots\}$. $C^n$ with $n\in\mathbb{N}_0$ stands for the set of $n$-times differentiable functions. The norm $|\cdot|$ denotes the Euclidean norm.  A compact set denoted by $\mathcal{S}_{x^*}^{\delta}$ with $\delta \in (0,\infty)$ and $x^*\in\mathbb{R}^n$ is defined as $\{ x\in \mathbb{R}^n: |x-x^*|\le\delta \}$.
The gradient of a function $J\in C^2:\mathbb{R}^n\rightarrow \mathbb{R}$ is represented by $\gradx{J}{x}:=(\sfrac{\partial J}{\partial x}(x))^\top$ and its Hessian by $\nabla^2_x J(x) := \sfrac{\partial^2 J}{\partial x^2}(x)$.
A function $f(\epsilon):\mathbb{R}\rightarrow\mathbb{R}^n$ is said to be of order $\mathcal{O}(\epsilon)$, if there exist $k,\bar{\epsilon}\in\mathbb{R}$ such that $|f(\epsilon)| \le k\epsilon$, for all $\epsilon \in [0,\bar{\epsilon}]$. 
The operator $k \bmod n$ denotes the modulo operation.

\section{MAIN IDEA AND ALGORITHM}\label{sec:prob_statement}
The aim of this section is to motivate and to develop the basic steps of the proposed derivative-free optimization algorithms for unconstrained multidimensional optimization problems of the form
\begin{align}
\min_{x \in \mathbb{R}^n} J(x),
\label{eq:minJ}
\end{align}
where $J:\mathbb{R}^n \rightarrow \mathbb{R}$ is the so-called objective or cost function.

\subsection{Non-Commutative Vector Fields and Continuous-time Gradient Approximation}
Consider the scalar continuous-time system
\begin{align}
	\dot{x}(t) = f_1(x(t))u_1(t)+f_2(x(t))u_2(t)
	\label{eq:basic_ESS}
\end{align} 
with state $x(t)\in \mathbb{R}$, vector fields $f_1,f_2\in C^2:\mathbb{R}\rightarrow \mathbb{R}$
and periodic inputs $u_i(t) = u_i(t+T) \in \mathbb{R}$, $i=1,2$, with period $T\in \mathbb{R}$.
It is very well-known that, if for example 
periodic rectangular-shaped input signals such as
\begin{align}
&u_1(t) = \begin{cases} 
~~1, & t\in  [0,h) \\
~~0, &  t\in [h,2h) \\
-1, &  t\in [2h,3h) \\
~~0, &  t\in [3h,4h) 
\end{cases}, 
\end{align}
\vspace*{-3mm}
\begin{align}
	 &u_2(t) = \begin{cases} 
	~~0, &  t\in [0,h) \\
	~~1, &  t\in [h,2h) \\
	~~0, &  t\in [2h,3h) \\
	-1, &   t\in [3h,4h) 
	\end{cases}
\end{align}
are applied to system \eqref{eq:basic_ESS}, as shown in~\Cref{fig:ncm}(a), then, for $h>0$ small, a second order Taylor expansion reveals
\begin{align}
	x(4h)-x(0) = h^2[f_1,f_2](x(0))+\mathcal{O}(h^{3}),
	\label{eq:brockett}
\end{align}
where $[f_1,f_2]=\frac{\partial f_2}{\partial x} f_1 -\frac{\partial f_1}{\partial x} f_2$ is the Lie bracket between
the vector fields $f_1,f_2$. Scaling the vector fields $f_1,f_2$ by $\frac{1}{\sqrt{h}}$ gives $x(4h)-x(0) = h[f_1,f_2](x(0))+\mathcal{O}(\sqrthdz)$. 
Then, by successively repeating \eqref{eq:brockett} with new initial values $x(lh),\ l=1,2,\ldots$ and taking the limit $h\rightarrow 0$ leads to the approximate (Lie bracket) system
\begin{align}
	\dot{\bar{x}}(t) = [f_1,f_2](\bar{x}(t)). 
	\label{eq:lie_sys}
\end{align}
Regarding the (scalar) optimization problem \eqref{eq:minJ}, one desirable methodology to
find an extremum, would be a simple continuous-time
 gradient descent equation $\dot{x} = -\nabla_xJ(x)$.  Interestingly, by defining 
\begin{align}
\label{fancyf}
 \fJ_i(x) := f_i(J(x)),
\end{align}
$i=1,2$, and choosing $\fJ_1,\fJ_2: \mathbb{R}\rightarrow \mathbb{R}$ as 
\begin{alignat}{2}
	&\fJ_1(x) = J(x),\qquad &&\fJ_2(x)=1,\quad \text{or} \label{eq:vec_simple} \\
	&\fJ_1(x) = \sin(J(x)),\quad &&\fJ_2(x) = \cos(J(x)),
	\label{eq:vec_sincos}
\end{alignat}
we have $[\fJ_1,\fJ_2](x) = -\nabla_xJ(x)$. Hence, by substitution $f_i$ by $h^{-1/2}\fJ_i$, $i=1,2$,
in \eqref{eq:brockett} and \eqref{eq:basic_ESS}, respectively,
we obtain by \eqref{eq:basic_ESS} a derivative-free approximation of a gradient flow.
This observation has been utilized and generalized for the design and analysis of continuous-time extremum seeking algorithms, see \cite{durr2013lie,grushkovskaya2017class,scheinker2014extremum} and references therein. Especially, \cite{grushkovskaya2017class} presents a broad class of generating vector fields $\fJ_1,\fJ_2$ for gradient approximation.
Non-commutativity comes into play since Lie brackets naturally arise when studying the commutativity of flows. In particular, if we denote with $\varphi_{f_i}^{t}(x_0):\mathbb{R}^n\rightarrow \mathbb{R}^n$ the flow map of the vector field $f_i:\mathbb{R}^n\rightarrow\mathbb{R}^n$ at time $t+t_0$ and initial condition $x(t_0)=x_0\in\mathbb{R}^n$, then we can express equation \eqref{eq:brockett} as
\begin{align}
		x(4h) = \left(\varphi_{-f_2}^h \circ \varphi_{-f_1}^h \circ \varphi_{f_2}^h \circ \varphi_{f_1}^h\right)(x_0),
	\label{eq:dtes}
\end{align}
as shown in~\Cref{fig:ncm}(b). It is easy to see that $x(4h)=x(0)$ if and only if the flow maps $\varphi_{f_1}^h$,$\varphi_{f_2}^h$ commute, since the flow is a bijection with $(\varphi_{f_1}^h)^{-1}=\varphi_{-f_1}^h$ and hence 
$x(4h)=x(0)$ if and only if $(\varphi_{f_2}^h \circ \varphi_{f_1}^h)(x(0))=(\varphi_{f_1}^h \circ \varphi_{f_2}^h)(x(0))$.
Further, it can be shown that the flow maps commute if and only if the Lie bracket between $f_1,f_2$ vanishes \cite{nijmeijer1990nonlinear}, hence the Lie bracket is an infinitesimal measure for the commutativity of vector fields.

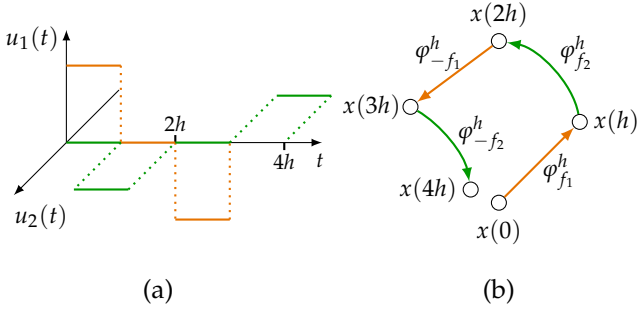
\begin{figure}[t]
	\centering
	\begin{tikzpicture}[>=latex]
	\node[draw=black,circle,name=c1, inner sep = 2pt] at (0,0) {};
	\node[draw=black,circle,name=c2, inner sep = 2pt, above right = 1cm and 1cm of c1.center ] {};
	\node[draw=black,circle,name=c3, inner sep = 2pt,above left = 1cm and 1cm of c2.center ] {};
	\node[draw=black,circle,name=c4, inner sep = 2pt, below left = 0.8cm and 1.1cm of c3.center ] {};
	\node[draw=black,circle,name=c5, inner sep = 2pt, above left = 0.1cm and 0.3cm of c1.center ] {};
	%
	%
	\draw[->,thick,color = orange90] (c1) to (c2);
	\draw [->,thick,color = green60] (c2) to [out=100,in=350,looseness=0.8] (c3);
	\draw [->,thick,color = orange90] (c3) to (c4);
	\draw [->,thick,color = green60] (c4) to [out=330,in=100,looseness=0.8] (c5);

	\node[below = -0.01cm of c1.south ] {\small{$x(0)$}};
	\node[right = -0.05cm of c2.east ] {\small{$x(h)$}};
	\node[above = -0.05cm of c3.north ] {\small{$x(2h)$}};
	\node[left = -0.1cm of c4.west ] {\small{$x(3h)$}};
	\node[left = -0.05cm of c5.west ] {\small{$x(4h)$}};
	\node[below = 1cm of c1.north] {(b)};
	\node[below left = 1cm and 4.2cm of c1.north] {(a)};
	
	\node[above right = 0.05cm and 0.45cm of c1.center ] {\small{$\varphi_{f_1}^{h}$}};
	\node[above left = 0.55cm and -0.3cm of c2.center ] {\small{$\varphi_{f_2}^{h}$}};
	\node[below left = -0.2cm and 0.35cm of c3.center ] {\small{$\varphi_{-f_1}^{h}$}};
	\node[below right = -0.01cm and 0.5cm of c4.center ] {\small{$\varphi_{-f_2}^{h}$}};
	
	\node[name=1, inner sep = 0pt] at (-5.75,0.8) {};
	\node[name=2, right = 3.4cm of 1.center ] {};
	\node[name=3, above = 1.5cm of 1.center ] {};
	\node[name=4, below = 1.0cm of 1.center ] {};
	\node[name=5, above = 0.9cm of 1.center ] {};
	\node[name=6, right = 0.6cm of 5.center ] {};
	\node[name=7, right = 0.6cm of 1.center ] {};
	\node[name=8, right = 0.6cm of 7.center ] {};
	\node[name=9, below = 0.9cm of 8.center ] {};
	\node[name=10, right = 0.6cm of 9.center ] {};
	\node[name=11, above = 0.9cm of 10.center ] {};
	
	\node[name=12, right = 0.6cm of 1.center ] {};
	\node[name=13, below left = 0.5cm and 0.5cm of 12.center ] {};
	\node[name=14, right = 0.6cm of 13.center ] {};
	\node[name=15, above right = 0.5cm and 0.5cm of 14.center ] {};
	\node[name=16, right = 0.6cm of 15.center ] {};
	\node[name=17, above right = 0.5cm and 0.5 of 16.center ] {};
	\node[name=18, right = 0.6cm of 17.center ] {};
	\node[name=19, below left = 0.5cm and 0.5cm of 18.center ] {};
	
	\node[name=20, below left = 0.7cm and 0.7cm  of 1.center ] {};
	\node[name=21, above right = 0.7cm and 0.7cm  of 1.center ] {};
	
	%
	\draw[thick] (8.center) to ([yshift=0.1cm]8.center);
	\draw[thick] (19.center) to ([yshift=-0.1cm]19.center);
	\draw[->] (1.center) to (2);
	\draw[->] (1.center) to (3);
	\draw[->] (1.center) to (20);
	\draw (1.center) to (21);
	\draw[thick,color = orange90] (5.center) to (6.center);
	\draw[dotted,thick,color = orange90] (6.center) to (7.center);
	\draw[thick,color = orange90] (7.center) to (8.center);
	\draw[dotted,thick,color = orange90] (8.center) to (9.center);
	\draw[thick,color = orange90] (9.center) to (10.center);
	\draw[dotted,thick,color = orange90] (10.center) to (11.center);
	
	\draw[thick,color = green60] (1.center) to (12.center);
	\draw[dotted,thick,color = green60] (12.center) to (13.center);
	\draw[thick,color = green60] (13.center) to (14.center);
	\draw[dotted,thick,color = green60] (14.center) to (15.center);
	\draw[thick,color = green60] (15.center) to (16.center);
	\draw[dotted,thick,color = green60] (16.center) to (17.center);
	\draw[thick,color = green60] (17.center) to (18.center);
	\draw[dotted,thick,color = green60] (18.center) to (19.center);

	\node[below left = -0.05cm and -0.15cm of 3.west ] {\small{$u_1(t)$}};
	\node[below right = -0.25cm and 0.01 cm of 20.south ] {\small{$u_2(t)$}};
	\node[above = -0.1cm of 8.north ] {\footnotesize{$2h$}};
	\node[below = -0.1cm of 19.south ] {\footnotesize{$4h$}};
	\node[below right = -0.1cm and 0.3 cm of 19.south ] {\footnotesize{$t$}};

	\end{tikzpicture}
	\caption{Switching vector fields: (a) periodic inputs $u_1(t),u_2(t)$ depicted for one period $4h$; (b) non-commutative flow maps of vector fields $\pm f_1,\pm f_2$.}
	\label{fig:ncm}
\end{figure}

\subsection{Non-Commutative Maps}

In principle, one could just numerically integrate \eqref{eq:basic_ESS} with appropriate vector fields and inputs in order to get a derivative-free discrete-time algorithm. However, rectangular-shaped inputs as shown in~\Cref{fig:ncm}(a) often lead to bad numerical behavior in combination with standard numerical integration schemes. Hence, in continuous-time algorithms \cite{durr2013lie, grushkovskaya2017class}, the rectangular-shaped inputs are often replaced by sinusoidal inputs in order to avoid these problems.  
A key idea of the proposed discrete-time optimization algorithm is to incorporate non-commutative maps more directly
into a tailored integration scheme, which approximates a gradient descent step with the help of composite maps of the form \eqref{eq:dtes}.
Hereby, two challenges arise: i) How to efficiently obtain the maps? 
ii) What are suitable discrete-time input functions in an $n$-dimensional optimization problem?

i) Maps. We construct the composition of flow maps by suitable numerical integration methods, where we consider two approaches. 
Firstly, we utilize the  \textit{Euler}-integration method.
It does not lead to a Lie bracket approximation as in \eqref{eq:brockett}, 
since the Euler-integration method is a first order method while Lie brackets are second order
effects. However, we show that the Euler-integration method still can be used for gradient approximation.
1Secondly, we utilize the so-called \textit{Heun}-integration method, also known as the trapezoidal-integration method, which preserve properties like \eqref{eq:brockett} (see~\Cref{lem:algo} in~\Cref{sec:main_part}). 
Accordingly, for an ODE $\dot{x}=g(x)$ with $x\in\mathbb{R}^n$ and vector field $g:\mathbb{R}^n\rightarrow \mathbb{R}^n$ we use the Euler- and Heun-integration steps given by
\begin{align}
E_g^h(x_k)&:= x_k + hg(x_k), \quad \text{and}
\label{eq:eulerint} \\
H_g^h(x_k) &:= x_k + \frac{h}{2}\left(g(x_k)+g(x_k+hg(x_k))\right),
\label{eq:heun_int}
\end{align}
as an approximation of the flow map $\varphi_g^h$.

ii) Inputs. As depicted in~\Cref{fig:ncm}(a), we consider periodic rectangular-shaped inputs. 
These inputs can simply be interpreted as switching  $\fJ_1,\fJ_2:\mathbb{R}^n\rightarrow \mathbb{R}$ at every time step (interval) $k$ ($h$). Hence, a switched vector field is defined as 
\begin{align}
	g_k(x) = \begin{cases}
	~~\fJ_1(x)e_l &\text{if}\ k\bmod 4 = 0 \\
	~~\fJ_2(x)e_l &\text{if}\ k\bmod 4 = 1 \\
	-\fJ_1(x)e_l & \text{if}\ k\bmod 4 = 2 \\
	-\fJ_2(x)e_l & \text{if}\ k\bmod 4 = 3 
	\end{cases},
	\label{eq:seq_switch_seq_n}
\end{align}
where $l = k/4 \bmod n+1$ with $e_l$ the $l$-th unit vector in $\mathbb{R}^n$. Functions $\fJ_1,\fJ_2:\mathbb{R}^n\rightarrow \mathbb{R}$, as for example in \eqref{eq:vec_simple} and \eqref{eq:vec_sincos} are components of vector fields $g_k(x) = \fJ_i(x)e_l$, but for the sake of convenience we sometimes call $\fJ_1,\fJ_2$ vector fields and use the notation $[\fJ_1,\fJ_2]=\nabla_x\fJ_2 \fJ_1-\nabla_x\fJ_1 \fJ_2$. Regarding \eqref{eq:seq_switch_seq_n}, one can think of executing four steps in every coordinate direction successively, hence the switching rule represents a coordinate wise approximation scheme. To keep the main idea clear, we consider only this single switching function and refer to further work where we will investigate more complex switching policies. 

\subsection{Derivative-free Optimization Algorithms}
Combining the integration steps \eqref{eq:eulerint} and \eqref{eq:heun_int}, respectively, with the sequential switching rule \eqref{eq:seq_switch_seq_n} and scaling 
the vector fields $\fJ_1,\fJ_2$ with $\nsqrth$, as discussed in the previous section, results in our discrete-time derivative-free optimization algorithms:
\begin{subnumcases} {x_{k+1} = M^{\sqrth}_{g_k}(x_k) = \label{eq:algos}}
\hspace{0pt}E_{g_k}^{\sqrth}(x_k) \label{eq:algo_euler}  \\
\hspace{0pt}H_{g_k}^{\sqrth}(x_k). \label{eq:algo}
\end{subnumcases}
As shown in the next section, every four steps (scalar case)
\begin{align}
	x_{k+4} = (M^{\sqrth}_{g_{k+3}} \circ M^{\sqrth}_{g_{k+2}}  \circ  M^{\sqrth}_{g_{k+1}} \circ M^{\sqrth}_{g_{k}} )(x_k)
	\label{eq:algo_flow}
\end{align}
a gradient descent step under suitable conditions on the vector fields $\fJ_1,\fJ_2$ is approximated. 
Summarizing, two algorithms presented in \eqref{eq:algos} and also in~\Cref{alg:heun4}, were derived by utilizing the effect of non-commutative maps, introduced by a sequential switching rule of vector fields and proper integration methods.
In the next section we analyze both schemes by showing the gradient approximation property and prove asymptotic convergence to a neighborhood of an extremum.  

\begin{algorithm}\label{alg:four_substeps}
	\caption{Derivative-free optimization algorithm with non-commutative maps \eqref{eq:algo}}
	\begin{algorithmic}[1]
		\State \textbf{Required:} $x_0$,$h$,$\fJ_1,\fJ_2$, stop criterion
		\State \textbf{Init:} $k=0$
		\While {stop criterion is not fulfilled}
		\State $l \leftarrow n \bmod k/4 +1$
		\State $e_l \leftarrow [0_{l},1,0_{n-1-l}]^\top$
		\State
		$g_k(x) \leftarrow 
		\begin{cases}
		\ \  \fJ_1(x)e_l	&\textbf{if }k\bmod 4= 0 \\ 
		\ \  \fJ_2(x)e_l	&\textbf{if }k\bmod 4 = 1 \\ 
		-\fJ_1(x)e_l	&\textbf{if }k\bmod 4 = 2 \\ 
		-\fJ_2(x)e_l 	&\textbf{if }k\bmod 4 = 3
		\end{cases}$
		\State $c_1 \leftarrow  g_k(x_k)$
		\State $c_2 \leftarrow  g_k(x_k+\sqrth c_1)$
		\State $x_{k+1} \leftarrow x_k + \frac{\sqrth}{2}(c_1+c_2)$
		\State $k\leftarrow k+1$
		\EndWhile
		\State \textbf{return} \textit{$[x_{0},x_1,\ldots]$}
	\end{algorithmic}
\label{alg:heun4}
\end{algorithm}

\section{MAIN RESULTS}\label{sec:main_part}
In the sequel, basic convergence and approximation properties of both 
discrete-time derivative-free optimization
algorithms presented in \eqref{eq:algos} are analyzed. Note, convergence rate results of derivative-free algorithms are almost absent in literature \cite{nesterov2017random}.
The  algorithms are applicable for broad classes of objective functions $J$, including  non-differentiable objective functions, given a choice of proper algorithmic parameters 
$x_0,h,\fJ_1,\fJ_2$, as we observe in simulations,
since only function evaluations of $J$ are necessary.
However, for the approximation and convergence analysis we assume:
\begin{assumption}
	\label{A1}
	The objective function $J:\mathbb{R}^n\rightarrow\mathbb{R}$ in \eqref{eq:minJ}
		and the vector fields $\fJ_1,\fJ_2:\mathbb{R}^n \rightarrow \mathbb{R}$ in \eqref{fancyf}		
		are class $C^2$ functions. $\hfill {\footnotesize \blacktriangle}$
\end{assumption}
First, we state two lemmas, which show the approximated optimization direction after $4n$ steps, induced by the non-commutative maps of our proposed algorithms \eqref{eq:algos}: 
\begin{lem}
	\label{lem:algo_euler}
	Let~\Cref{A1} hold and consider algorithm \eqref{eq:algo_euler}. Then the evolution of $x_k\in\mathbb{R}^n$ at step $k+4n$ is given by  
	\begin{align}
	x_{k+4n} = x_k  + h & \left(  [\fJ_1,\fJ_2](x_k) - \nabla_x\fJ_1(x_k)\fJ_1(x_k) \right. \nonumber \\ & \left.- \nabla_x\fJ_2(x_k)\fJ_2(x_k)\right)  +  \mathcal{O}(\sqrthdz).
	\label{eq:lem_algo_euler}
	\end{align}
	for all $k\in \mathbb{N}_{0}$
	$\hfill {\footnotesize \blacktriangle}$
\end{lem}

\begin{lem}
	\label{lem:algo}
	Let~\Cref{A1} hold and consider algorithm \eqref{eq:algo}. Then the evolution of $x_k\in\mathbb{R}^n$ at step $k+4n$ is given by 
	\begin{align}
		x_{k+4n} = x_k + h [\fJ_1,\fJ_2](x_k) + \mathcal{O}(\sqrthdz).
		\label{eq:lem_algo}
	\end{align}
	for all $k\in \mathbb{N}_{0}$ $\hfill {\footnotesize \blacktriangle}$
\end{lem}

The proofs of~\Cref{lem:algo_euler} and~\Cref{lem:algo} are given in~\Cref{apx:algo_proof_euler} and~\Cref{apx:algo_proof}, respectively.
We observe that the Heun-integration steps $H_{g_k}^{\sqrth}(x_k)$ preserve the Lie bracket approximation property, as shown in \eqref{eq:lem_algo}. 
However, as mention in Section 1, successively applying the Euler-integration steps $E_{g_k}^{\sqrth}(x_k)$ reveals the Lie bracket but additional terms of the same order $\mathcal{O}(h)$  are present, as shown in \eqref{eq:lem_algo_euler}. 
Interestingly, choosing proper vector fields $\fJ_1,\fJ_2$, \emph{both} algorithms approximate a gradient descent method:
\begin{theorem}
	\label{cor:euler}
	Let~\Cref{A1} hold and consider the algorithms in \eqref{eq:algos}
	with the  pair of vector fields \eqref{eq:vec_sincos}.  
	Then the evolution of $x_k\in\mathbb{R}^n$ at step $k+4n$ is given by 
	\begin{align}
			x_{k+4n} = x_k - h\nabla_xJ(x_k)+\mathcal{O}(\sqrthdz)
		\label{eq:grad_desc}
	\end{align}
	for all $k\in \mathbb{N}_{0}$ $\hfill {\footnotesize \blacktriangle}$
\end{theorem}
\begin{proof}
	This result follows directly from~\Cref{lem:algo_euler} and~\Cref{lem:algo} and simple calculations.
\end{proof}
\begin{rem}
	The pair of vector fields \eqref{eq:vec_simple} with algorithm \eqref{eq:algo_euler}, as commonly employed in continuous-time methods \cite{durr2013lie}, yield an approximation scheme of the form
	\begin{align}
		x_{k+4n} = x_k - h \nabla_x J(x_k) (1+J(x_k)) + \mathcal{O}(\sqrthdz).
	\end{align}
	Hence, this scheme is not a gradient descent step but it is still suitable for optimization if, for example, the objective function is positive semi-definite. $\hfill {\footnotesize \blacktriangle}$
\end{rem}

For the convergence analysis of \eqref{eq:algos}, we impose the following additional assumptions on the cost function $J:\mathbb{R}^n\rightarrow\mathbb{R}$ in \eqref{eq:minJ} 
:
%
%
%
%
\begin{assumption} 
	\label{A2}
	$J(x)$ is radially unbounded and there exists a $x^*\in\mathbb{R}^n$ such that $\gradx{J}{x}^\top (x-x^*)>0$ for all $x\in\mathbb{R}^n \backslash \{x^*\}$. 
	$\hfill {\footnotesize \blacktriangle}$
\end{assumption}
\begin{rem}
	Assumption A2 implies that $x^*$ is the unique global minimizer and there exists no other local minimum. 
	
	$\hfill {\footnotesize \blacktriangle}$
\end{rem}
%
%

The following theorem shows asymptotic convergence to a neighborhood of the minimum $x^*$ of $J(x)$:
\begin{theorem}
	\label{thm:conv}
	Let~\Cref{A1} and~\Cref{A2} hold.
	Consider the discrete-time derivative-free optimization algorithm \eqref{eq:algos} and let the vector fields $\fJ_1,\fJ_2$ satisfy for \eqref{eq:algo_euler} and \eqref{eq:algo} respectively:
	\begin{align}
	-\nabla_xJ(x)&=\left([\fJ_1,\fJ_2] -\nabla_x\fJ_1\fJ_1 - \nabla_x\fJ_2\fJ_2\right)(x), \label{eq:thm2_1} \\
		-\nabla_xJ(x)&=[\fJ_1,\fJ_2](x).\label{eq:thm2_2}
	\end{align}			
	Then for all $0<\delta_2<\delta_1$ there exists an $h^*>0$ such that for all $h\in (0,h^*)$ and  all initial conditions $x_0 \in \mathcal{S}_{x^*}^{\delta_1}$,	$x_k$ converges to $\mathcal{S}_{x^*}^{\delta_2}$.
	$\hfill {\footnotesize \blacktriangle}$
\end{theorem}

The proof of~\Cref{thm:conv} is given in~\Cref{apx:thm_conv_proof}.
Note, that~\Cref{thm:conv} states the semi-global uniform practical asymptotic stability property \cite{durr2013lie} of $\mathcal{S}_{x^*}^{\delta_2}$ of \eqref{eq:algos} under the given assumptions.

\section{SIMULATIONS}\label{sec:simulation}
In this section we study our discrete-time derivative-free optimization algorithms \eqref{eq:algos} in various simulation examples and provide some tuning rules. To this end, we consider the cost function $J(x) = (x - 2)^2 + 6$. If not otherwise specified we use the vector fields \eqref{eq:vec_sincos} and initialize the algorithms with $x_0=x(0) = 0.5$. For the sake of convenience we define this setup as $\mathcal{P}:=\left(J,x_0,\fJ_1,\fJ_2\right)$. Note that the results obtained with such a quadratic cost function also provide insight into the results that would be obtained with an arbitrary function of class $C^2$ near a minimum
with positive definite Hessian.
We start by analyzing \eqref{eq:algos}, applied to the basic setup $\mathcal{P}$, which is depicted for (a large) step size $h=0.5$ in~\Cref{fig:sim_dES}, where the switching of the vector fields $\fJ_1,\fJ_2$ is highlighted. 
\pgfplotsset{minor grid style={dotted}}
As it can be observed, the trajectories of $x_k$ are converging into a neighborhood of $x^*$, which is of order $\mathcal{O}(\sqrth)$ (see proof of~\Cref{thm:conv}).
To eliminate the steady-state oscillatory behavior, we propose the filter $y_{k+1} = y_k + \frac{1}{4n}\sum_{i=k-4n}^{k-1} x_{i+1}-x_{i}$
where $y_0= x_0$ and $x_l = x_0,\ l<0$, which shows a good asymptotic convergence behavior of the algorithms \eqref{eq:algos} to the extremum $x^*$ as illustrated in~\Cref{fig:sim_dES}. 
\begin{figure}[t]
	\centering
	\setlength\figureheight{3.9cm}
	\setlength\figurewidth{7.5cm}
	{\small \input{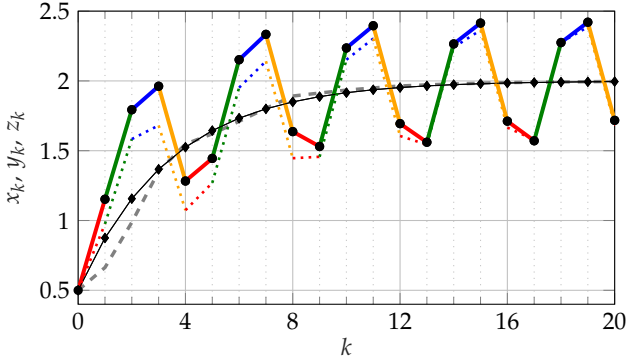}}
	\caption[]{Convergence of $x_k$ in \eqref{eq:algos}  for setup $\mathcal{P}$ with $h=0.5$ to a neighborhood of $x^*=2$. The flow maps as in \eqref{eq:algo_flow} for $M=E$ are highlighted with $E_{\fJ_1}^{\sqrth}$ ({\color{red}$\boldmin$}), $E_{\fJ_2}^{\sqrth}$ ({\color{green}$\boldmin$}),  $E_{-\fJ_1}^{\sqrth}$ ({\color{blue}$\boldmin$}), $E_{-\fJ_2}^{\sqrth}$ ({\color{orange}$\boldmin$}) and the same color scheme (dotted) for \eqref{eq:algo_flow}, $M=H$. This is compared with the trajectory resulting from a gradient descent algorithm with the exact gradient ($\thinmin$). By filtering $x_k$, a good asymptotic convergence behavior of $y_k$ to $x^*$ can be observed ({\color{gray}$\thinmins\,\thinmins$}).}
	\label{fig:sim_dES}
\end{figure}
Obviously, $y_k$ is a windowed filter with length $4n$, which is not fed back into the algorithm.
This behavior can also be observed in a multidimensional optimization problem as depicted in~\Cref{fig:sim_dES2D}.  
\begin{figure}[t]
	\centering
	\setlength\figureheight{5.5cm}
	\setlength\figurewidth{7cm}
	{\small \input{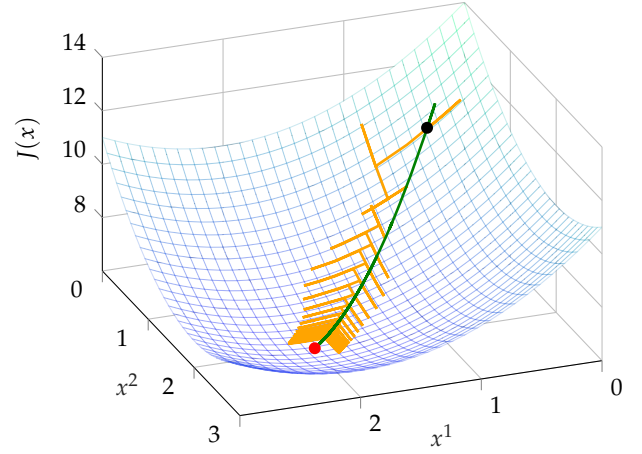}}
	\caption[]{A two-dimensional optimization problem with $J(x)=(x^1-2)^2+(x^2-2)^2+6$. Trajectories $x_k$ ({\color{orange}$\boldmin$}) of \eqref{eq:algo_euler} and $y_k$ ({\color{green}$\boldmin$}) converge to the extremum $x^*=[2,2]$ ($\red{\bullet}$) initialized with $x_0=y_0=[0.5,0.5]$ ($\bullet$).}
	\label{fig:sim_dES2D}
\end{figure}

As known from continuous-time methods, see for example \cite{durr2013lie}, increasing the (switching) frequency, and hence reducing the step size $h$, leads to a convergence into a smaller neighborhood around $x^*$ as depicted for different $h$'s in~\Cref{fig:sim_dES_diff_h}.
Clearly, for decreasing the step size, more iterations are needed, hence more function evaluations.
\begin{figure}[t]
	\centering
	\setlength\figureheight{3.9cm}
	\setlength\figurewidth{7.5cm}
	{\small \input{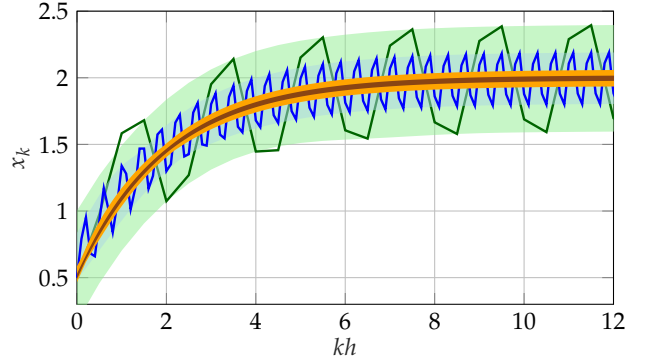}}
	\caption[]{Choosing smaller step sizes for algorithm \eqref{eq:algo_euler} or \eqref{eq:algo} yield to a smaller neighborhood of convergence. Trajectories of $x_k$ for setup $\mathcal{P}$ with step sizes $h=0.5$ ({\color{green}$\boldmin$}), $h=0.1$ ({\color{blue}$\boldmin$}), $h=0.01$ ({\color{orange}$\boldmin$}), $h=0.001$ ({\color{brown}$\boldmin$}).}
	\label{fig:sim_dES_diff_h}
\end{figure}

As mentioned in~\Cref{sec:intro}, an advantage of the proposed class of algorithms is robustness with respect to noisy objective functions.
Therefore, consider the setup $\mathcal{P}$, but with a noisy cost function $\bar{J}:\mathbb{R}\rightarrow \mathbb{R}$ such that
\begin{align}
	\bar{J}(x_k) \sim \mathcal{N}(J(x_k),0.04),
\end{align} 
where $\mathcal{N}(J(x_k),0.04)$ is a normal distribution with mean $J(x_k)$ and standard derivation $\sqrt{0.04}=0.2$. In~\Cref{fig:sim_dES_noise}, algorithm \eqref{eq:algo} is compared with the well known difference quotient method to approximate the gradient $D(J(\bar{x})):= h^{-1}(J(\bar{x}_k)-J(\bar{x}_k+h))$ used in a gradient descent algorithm $\bar{x}_{k+1}=\bar{x}_k-hD(J(\bar{x}))$ with step size $h = 0.01$ and initial condition $\bar{x}_0 = x_0$. 
\begin{figure}[t]
	\centering
	\setlength\figureheight{3.9cm}
	\setlength\figurewidth{7.5cm}
	{\small \input{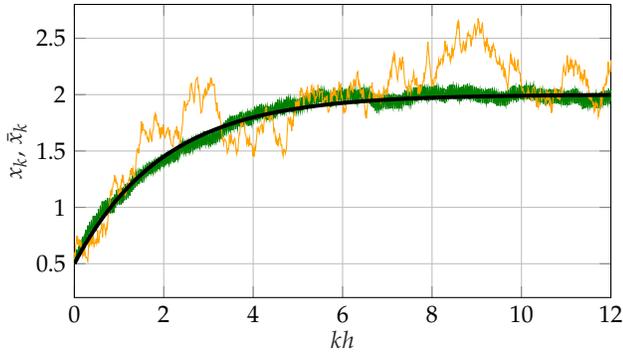}}
	\caption[]{Study of noisy objective function: trajectory $x_k$ of algorithm \eqref{eq:algo} ({\color{green}$\boldmin$}), trajectory $\bar{x}_k$ of gradient descent with gradient approximation via difference quotient ({\color{orange}$\boldmin$}), trajectory of gradient descent with exact gradient ({\color{black}$\boldmin$}).}
	\label{fig:sim_dES_noise}
\end{figure}
Apparently, approximating the gradient via the difference quotient method is much more sensitive regarding noisy objectives as algorithms \eqref{eq:algos}. This is reasoned by the averaging behavior of the proposed algorithms, implied by utilizing non-commutative maps. 

Concluding, we illustrated both algorithms \eqref{eq:algos} on simple examples and we introduced a filtering method. Secondly, we showed the impact of decreasing $h$ and the performance under noisy measurements. Regarding the convergence behavior, algorithm \eqref{eq:algo_euler} and \eqref{eq:algo} show similar convergence behavior for the considered vector fields \eqref{eq:vec_sincos} and switching function \eqref{eq:seq_switch_seq_n}. Nevertheless, algorithm \eqref{eq:algo} needs two times more function evaluations than \eqref{eq:algo_euler}. 

\section{CONCLUSION}\label{sec:conclusions}
In this work we introduced a methodology for the design of novel discrete-time derivative-free optimization algorithms. Based on non-commutative maps, designed by two different 
numerical integration methods and a sequential switching rule, it was shown how to approximate a gradient descent step. Furthermore, practical convergence and stability towards an extremum of the objective function was proven. The algorithms are illustrated using numerical examples. 

The introduced class of algorithms raise several new research questions, which will be part of future work. For example, what are more efficient switching rules in multidimensional optimization problems compared to the introduced coordinate wise switching rule. Moreover, it is interesting to characterize a general class of non-commutative maps which show gradient approximation properties or to study convergence rates or accelerated gradient schemes.

\bibliographystyle{IEEEtran}
\bibliography{IEEEabrv,bibfile}

\appendix
\section{APPENDIX}
\subsection{Proof of~\Cref{lem:algo_euler}}
\label{apx:algo_proof_euler}
We consider the derivative-free optimization algorithm \eqref{eq:algo_euler} and 
assume without loss of generality $k\bmod 4=0$ and $k/4\bmod n = 0$.
Due to our assumption on $k$, $g_k(x_k) = \fJ_1(x_k)e_1$. Since, the switching policy \eqref{eq:seq_switch_seq_n} is based on a coordinate-wise implementation, it is sufficient to analyze the evolution of $x_k$ in only one dimension. Hence, it holds
\begin{align}
	x_{k+1} &= x_k +\sqrth \fJ_1(x_k)e_1 .
\end{align}
 For the sake of readability in the sequel we define $\fJ_{i}^k := \fJ_i(x_k)e_1$ and $\FJ_{i}^k := (\frac{\partial \fJ_i}{\partial x}(x_k))^\top e_1$ and neglect $e_1$. 
 Therefore, the next step is given by 
 \begin{align}
 		x_{k+2} &= x_k + \sqrth (\fJ_1^k+\fJ_2(x_k+\sqrth \fJ_1^k)).
 	\label{eq:Apx1_1}
 \end{align}
Performing a Taylor expansion on $\fJ_2(x_k+\sqrth \fJ_1^k)$ in \eqref{eq:Apx1_1} reveals
\begin{align}
x_{k+2} = x_k + \sqrth(\fJ_1^k+\fJ_2^k)+hF_2^k\fJ_1^k+\mathcal{O}(\sqrthdz),
\label{eq:Apx1_2}
\end{align}
where all higher order terms are pushed in $\mathcal{O}(\sqrthdz)$. Repeating this procedure, as presented above, $x_k$ evolutes as follows:
\begin{align}
x_{k+3} &= x_k + \sqrth \fJ_2^k+h(F_2^k\fJ_1^k-F_1^k\fJ_1^k-F_1^k\fJ_2^k) \nonumber \\&+ \mathcal{O}(\sqrthdz), \\
x_{k+4} &= x_k +h(F_2^k\fJ_1^k-F_1^k\fJ_1^k-F_1^k\fJ_2^k-F_2^k\fJ_2^k)\nonumber \\&+\mathcal{O}(\sqrthdz).
	\label{eq:Apx1_5}
\end{align}
Eventually, repeating the same procedure for dimensions $\{2,3,\ldots,n\}$ delivers the same expression \eqref{eq:Apx1_5} for each dimension, hence \eqref{eq:lem_algo_euler} can be directly concluded.  
$\hfill \square$

\subsection{Proof of~\Cref{lem:algo}}
\label{apx:algo_proof}
As in the proof of~\Cref{lem:algo_euler} we repeat the same procedure and only state the main equations. For notations and assumptions (w.l.o.g) reread~\Cref{apx:algo_proof_euler}.
Therefore it holds
\begin{align}
	x_{k+1} = x_k + \frac{\sqrth}{2}\left(\fJ_1^k+ \fJ_1(x_k+\sqrth \fJ_1^k) \right).
	\label{eq:apx1_1}
\end{align}
Performing a Taylor expansion on $\fJ_1(x_k+\sqrth \fJ_1^k)$ in \eqref{eq:apx1_1} reveals
\begin{align}
x_{k+1} &=  x_k+\sqrth \fJ_1^k+\frac{h}{2}\jacob{F_1^k}\fJ_1^k+\mathcal{O}(\sqrthdz).
\label{eq:apx1_2}
\end{align}
 Repeat the procedure, as presented above, $x_k$ evolutes as follows:
\begin{align}
x_{k+2}&= x_k+\sqrth (\fJ_1^k+\fJ_2^k)\nonumber \\ 			&+\frac{h}{2}\left(\jacob{F_1^k}\fJ_1^k+\jacob{F_2^k}(2\fJ_1^k+\fJ_2^k)\right)+\mathcal{O}(\sqrthdz), \\
x_{k+3}&=x_k +\sqrth \fJ_2^k   + \mathcal{O}(\sqrthdz) \nonumber \\ 
&+ \frac{h}{2}(-2\jacob{F_1^k}\fJ_2^k+2\jacob{F_2^k}\fJ_1^k+\jacob{F_2^k}\fJ_2^k), \\
x_{k+4}&=x_k +  h(\jacob{F_2^k}\fJ_1^k-\jacob{F_1^k}\fJ_2^k) + \mathcal{O}(\sqrthdz) .
	\label{eq:apx1_5}
\end{align}
Eventually, repeating the same procedure for dimensions $\{2,3,\ldots,n\}$ delivers the same expression \eqref{eq:apx1_5} for each dimension, hence \eqref{eq:lem_algo} can be directly concluded.
 $\hfill \square$

\subsection{Proof of~\Cref{thm:conv}}
\label{apx:thm_conv_proof}
Let $0\leq\delta_3\leq\delta_2\leq\delta_1\leq \delta_0$. The proof is separated in: \linebreak[4]1) Define a Lyapunov-like function $V(x)$, such that $V(x_{k+4n})-V(x_k)<0$ for $\mathcal{S}_{x^*}^{\delta_0}\backslash \mathcal{S}_{x^*}^{\delta_3}$, 2) Practical invariance of $\mathcal{S}_{x^*}^{\delta_1}$, 3) Convergence of $x_k$ to $\mathcal{S}_{x^*}^{\delta_2}$ with $x_0 \in \mathcal{S}_{x^*}^{\delta_1}$.

1) Consider the algorithms \eqref{eq:algos} and their evolution 
as given in \eqref{eq:lem_algo_euler} and \eqref{eq:lem_algo} under conditions \eqref{eq:thm2_1}, \eqref{eq:thm2_2}  on vector fields $\fJ_1,\fJ_2$ as stated in~\Cref{thm:conv}. 
W.l.o.g., let the Lagrange remainder of the Taylor expansion  \eqref{eq:lem_algo_euler} and \eqref{eq:lem_algo}, respectively, be of the form $\sqrthdz R_k=\mathcal{O}(\sqrthdz)$, where $R_k\in\mathbb{R}^n$ depends on the vector fields $\mathfrak{f}_1,\mathfrak{f}_2$ and their Jacobians and Hessians. Hence, \eqref{eq:algos} is given by 
\begin{align}
x_{k+4n} = x_{k} - h \nabla_x J (x_{k})+\sqrthdz R_k.
\end{align}
Consider the Lyapunov-like candidate function $V(x_k) = J(x_k)-J(x^*)$, then it holds for all $k\ge 0$,
\begin{align}
V(x_{{k}+4n})&-V(x_{k}) = \nonumber \\
&=J\left(x_{k}-h\nabla_xJ(x_{k})+\sqrthdz R_k\right)-J(x_{k}) \nonumber \\
&=-h\nabla_xJ(x_{k})^\top\nabla_xJ(x_{k})\nonumber \\&~~~~~+ \sqrthdz\nabla_xJ(x_{k})^\top R_k+ \frac{h^2}{2}\gamma ^\top H_k\gamma,
\end{align}
where the last equality is gained by a second order Taylor expansion at $x_k$, 
with the Lagrange reminder $\frac{h^2}{2}\gamma^\top H_k \gamma$, where $\gamma = \nabla_x J(x_{k}) + \sqrthdz R_k$ and $ H_k\in \mathbb{R}^{n\times n}$. $R_k$ and $H_k$  depend on $J,\ \fJ_1,\ \fJ_2$ and their Jacobians and Hessians, Hence, $R_k,H_k$, and $\nabla_x J(x)$ are bounded on every compact set $\mathcal{S}_{x^*}^{\delta_0}$, due to~\Cref{A1}. Therefore, for every $\delta_0$ there exists a $0\leq R_{\delta_0} \in \mathbb{R}$ such that for all $x_{k} \in \mathcal{S}_{x^*}^{\delta_0}$ hold 
\begin{align}
V(x_{k+4n}) - V(x_{k}) \le -h|\nabla_x J(x_{k})|^2 + \sqrthdz R_{\delta_0}.
\end{align} Then under~\Cref{A2}, there exists a $h_1>0$, such that for all $h \in (0,h_1)$ hold 
\begin{align}	
V(x_{k+4n})-V(x_k)\leq-\epsilon \text{\ \ on\ \ } \mathcal{L}=\mathcal{L}_0\backslash \mathcal{L}_3,
\label{eq:lyap_2}
\end{align}
with $\epsilon>0$, $\mathcal{S}_{x^*}^{\delta_0}\backslash \mathcal{S}_{x^*}^{\delta_3} \subseteq \mathcal{L}$ and sub-level sets of $V(x)$, $\mathcal{L}_0 \supseteq \mathcal{S}_{x^*}^{\delta_0}$ and $\mathcal{L}_3 \subseteq \mathcal{S}_{x^*}^{\delta_3}$.

2) By~\Cref{A1} and the first order Taylor expansions, it holds that for all $k\geq 0$, $0<l< 4n$, there exist an upper bound on the Lagrange reminders $0\leq \tilde{R}_{k+l} \in \mathbb{R}$ in $\mathcal{S}_{x^*}^{\delta_0}$ such that 
\begin{align}
V(x_{k+l}) &= V(x_k + \sqrth \sum_{i=k}^{k+l-1}g_i(x_i)) \nonumber \\&\le V(x_k)+\sqrth \tilde{R}_{k+l}
\end{align}
Hence, in combination with 1) there exist a $h_2>0$ such that for all $h \in (0,h_2)$, $x_0 \in \mathcal{S}_{x^*}^{\delta_1} \Rightarrow x_k \in \mathcal{S}_{x^*}^{\delta_0}$ for all $k\geq 0$. 

3) Due to \eqref{eq:lyap_2} with $h_1>0$, it holds that for every initial value $x_0 \in \mathcal{S}_{x^*}^{\delta_1}$, there exist a maximum number of iterations $\bar{k}>0$, such that for all $h\in(0,\min\lbrace h_1,h_2\rbrace)$, it holds $x_{\bar{k}} \in \mathcal{S}_{x^*}^{\delta_3}$.
Furthermore, as in 2), one can show that there exists a $h_3>0$, such that for all $h\in(0,h_3)$ and $x_0 \in \mathcal{S}_{x^*}^{\delta_3},x_k \in \mathcal{S}_{x^*}^{\delta_2}$ for all $k\geq 0$. Thus, for all $h\in(0,\min\lbrace h_1,h_2,h_3\rbrace)$, if $x_0 \in \mathcal{S}_{x^*}^{\delta_1}$, then $x_k$ converges to $\mathcal{S}_{x^*}^{\delta_2}$.~\Cref{fig:proof} illustrates the proof and the introduced compact sets $\mathcal{S}_{x^*}$. \hfill $\square$

\begin{figure}[h]
	\centering
	\begin{tikzpicture}[>=latex]
	\node[fill=orange90!30,semicircle,name=c1, inner sep = 45pt] at (0,30pt) {};
	\node[fill=orange90!10,semicircle,name=c2, inner sep = 40pt] at (0,25pt) {};
	\node[fill=green60!30,semicircle,name=c3, inner sep = 20pt] at (0,5pt) {};
	\node[fill=green60!10,semicircle,name=c4, inner sep = 15pt] at (0,0) {};
		
	\node[fill =red, circle, name=t1, inner sep = 2pt] at (0.4,-0.1) {};
	\node[fill =black, circle, name=t2, inner sep = 1.5pt, above right = 0.3cm and -0.3cm of t1] {};
	\node[fill =black, circle, name=t3, inner sep = 1.5pt, above right = 0.2cm and -0.2cm of t2] {};
	\node[fill =black, circle, name=t4, inner sep = 1.5pt, above right = 0.2cm and -0.6cm of t3] {};
	\node[fill =black, circle, name=t5, inner sep = 1.5pt, above right = -0.4cm and -0.3cm of t4] {};
	\node[fill =black, circle, name=t6, inner sep = 1.5pt, above right = -0.2cm and -0.3cm of t5] {};
	\node[fill =black, circle, name=t7, inner sep = 1.5pt, above right = -0.4cm and -0.2cm of t6] {};
	\node[fill =black, circle, name=t8, inner sep = 1.5pt, above right = -0.4cm and 0.2cm of t7] {};
	\node[fill =red, circle, name=t9, inner sep = 2pt, above right = 0.2cm and 0.2cm of t8] {};
	\node[fill =black, circle, name=t10, inner sep = 1.5pt, above right = 0.6cm and 0.1cm of t9] {};
	\node[fill =black, circle, name=t11, inner sep = 1.5pt, above right = 0.4cm and -0.4cm of t10] {};
	\node[fill =black, circle, name=t12, inner sep = 1.5pt, above right = 0.1cm and -0.4cm of t11] {};
	\node[fill =black, circle, name=t13, inner sep = 1.5pt, above right = -0.1cm and -0.5cm of t12] {};
	\node[fill =black, circle, name=t14, inner sep = 1.5pt, above right = -0.4cm and -0.4cm of t13] {};
	\node[fill =black, circle, name=t15, inner sep = 1.5pt, above right = -0.5cm and -0.1cm of t14] {};
	\node[fill =black, circle, name=t16, inner sep = 1.5pt, above right = -0.1cm and 0.2cm of t15] {};
	\node[fill =red, circle, name=t17, inner sep = 2pt, above right = 0.0cm and 0.2cm of t16] {};
	\node[fill =black, circle, name=t18, inner sep = 1.5pt, above right = 0.6cm and 0.3cm of t17] {};
	\node[fill =black, circle, name=t19, inner sep = 1.5pt, above right = 0.4cm and 0.2cm of t18] {};
	\node[fill =black, circle, name=t20, inner sep = 1.5pt, above right = 0.2cm and -0.4cm of t19] {};
	\node[fill =black, circle, name=t21, inner sep = 1.5pt, above right = -0.1cm and -0.5cm of t20] {};
	\node[fill =black, circle, name=t22, inner sep = 1.5pt, above right = -0.4cm and -0.2cm of t21] {};
	\node[fill =black, circle, name=t23, inner sep = 1.5pt, above right = -0.3cm and 0.1cm of t22] {};
	\node[fill =black, circle, name=t24, inner sep = 1.5pt, above right = -0.3cm and 0.4cm of t23] {};
	\node[fill =green60, circle, name=t25, inner sep = 2pt, above right = 0.0cm and 0.2cm of t24] {};
	
	\node[] at (0,-15pt) {$\bm{\times}$} ;
	\node[] at (-0.25,-9pt) {$x^*$} ;
	\node[] at (0.9,2.2) {$x_0$} ;
	
	%
	\draw[thick] (t1.center) to (t2.center);
	\draw[thick] (t2.center) to (t3.center);
	\draw[thick] (t3.center) to (t4.center);
	\draw[thick] (t4.center) to (t5.center);
	\draw[thick] (t5.center) to (t6.center);
	\draw[thick] (t6.center) to (t7.center);
	\draw[thick] (t7.center) to (t8.center);
	\draw[thick] (t8.center) to (t9.center);
	\draw[thick] (t9.center) to (t10.center);
	\draw[thick] (t10.center) to (t11.center);
	\draw[thick] (t11.center) to (t12.center);
	\draw[thick] (t12.center) to (t13.center);
	\draw[thick] (t13.center) to (t14.center);
	\draw[thick] (t14.center) to (t15.center);
	\draw[thick] (t15.center) to (t16.center);
	\draw[thick] (t16.center) to (t17.center);
	\draw[thick] (t17.center) to (t18.center);
	\draw[thick] (t18.center) to (t19.center);
	\draw[thick] (t19.center) to (t20.center);
	\draw[thick] (t20.center) to (t21.center);
	\draw[thick] (t21.center) to (t22.center);
	\draw[thick] (t22.center) to (t23.center);
	\draw[thick] (t23.center) to (t24.center);
	\draw[thick] (t24.center) to (t25.center);
	
	\node[fill =red, circle, inner sep = 2pt] at (0.4,-0.1) {};
	\node[fill =red, circle, inner sep = 2pt, above right = 0.2cm and 0.2cm of t8] {};
	\node[fill =red, circle, inner sep = 2pt, above right = 0.0cm and 0.2cm of t16] {};
	\node[fill =green60, circle, inner sep = 2pt, above right = 0.0cm and 0.2cm of t24] {};
	
	\node[] at (2.4,-25pt) {$\mathcal{S}_{x^*}^{\delta_1}$};
	\node[] at (3.4,-25pt) {$\mathcal{S}_{x^*}^{\delta_0}$};
	\node[] at (1.4,-25pt) {$\mathcal{S}_{x^*}^{\delta_2}$};
	\node[] at (0.7,-25pt) {$\mathcal{S}_{x^*}^{\delta_3}$};
	\end{tikzpicture}
	\caption{Convergence of $x_k\in\mathbb{R}^2$ ($\bullet$) with initial value $x_0$  ({\color{green60}$\bm \bullet$}) $\in \mathcal{S}_{x^*}^{\delta_1}\subseteq \mathcal{S}_{x^*}^{\delta_0}$ to $\mathcal{S}_{x^*}^{\delta_2}$. For all $k>\tilde{k}$, $x_k$ stays in $\mathcal{S}_{x^*}^{\delta_2}$. Note, every $4n$-th step is marked with {\color{red}$\bm \bullet$}.}
	\label{fig:proof}
\end{figure}
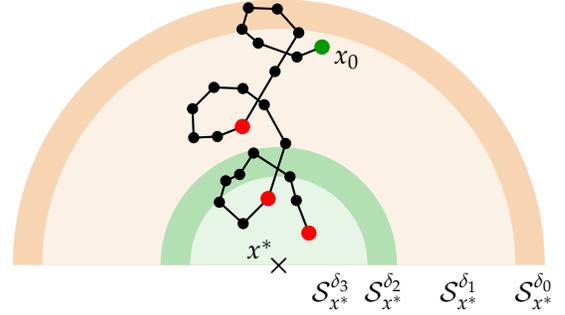

\end{document}